\documentclass[10pt,a4paper]{amsart} 
\usepackage{amsaddr}
\usepackage{amssymb,amsthm,amsmath,amsfonts,amscd}
\usepackage{graphicx}
\usepackage{hyperref} 
\usepackage{url}
\allowdisplaybreaks[1]
\usepackage[english]{babel}
\usepackage{enumerate}
\usepackage{bbm}
\usepackage{stmaryrd}
\usepackage{mathrsfs}

\newcommand{\norm}[1]{\left\lVert#1\right\rVert}

\newcommand{\abs}[1]{\lvert#1\rvert}
\newcommand{\lrabs}[1]{\left\lvert#1\right\rvert}

\newcommand{\lrbr}[2]{\left\langle#1,{#2}\right\rangle}

\newcommand{\lrdel}[2]{\left(#1,{#2}\right)}

\newcommand{\gdualdel}[4]{\sideset{_{#1}}{_{#2}}{\mathop{\left\langle{#3},{#4}\right\rangle}}}

\newcommand{\Fscr} {{\mathscr F}}

\newcommand{\Hscr} {{\mathscr H}}

\renewcommand{\tilde}{\widetilde}

\renewcommand{\d}{d}
\newcommand{\D}{\textup{D}}

\newcommand{\eps}{{\varepsilon}}
\renewcommand{\phi}{\varphi}

\newcommand{\R}{\mathbbm{R}}
\newcommand{\N}{\mathbbm{N}}

\renewcommand{\P}{\mathbbm{P}}
\newcommand{\E}{\mathbbm{E}}

\newcommand{\dom}{\operatorname{dom}}

\newcommand{\loc}{\textup{loc}}

\newcommand{\sgn}{\operatorname{sgn}}
\renewcommand{\le}{\leq}
\renewcommand{\ge}{\geq}

\renewcommand{\limsup}{\varlimsup}
\renewcommand{\liminf}{\varliminf}

\newcommand{\BIGOP}[1]{\mathop{\mathchoice%
{\raise-0.22em\hbox{\huge $#1$}}%
{\raise-0.05em\hbox{\Large $#1$}}{\hbox{\large $#1$}}{#1}}}

\newcommand{\BIGboxplus}{\mathop{\mathchoice%
{\raise-0.35em\hbox{\huge $\boxplus$}}%
{\raise-0.15em\hbox{\Large $\boxplus$}}{\hbox{\large $\boxplus$}}{\boxplus}}}

\def\leq {\leqslant}
\def\geq {\geqslant}

\numberwithin{equation}{section}

\newtheorem{theorem}{Theorem}[section] 

\theoremstyle{plain}
\newtheorem{defi}[theorem]{Definition}
\newtheorem{conj}[theorem]{Conjecture}
\newtheorem{lem}[theorem]{Lemma}

\newtheorem{prop}[theorem]{Proposition}
\newtheorem{rem}[theorem]{Remark}

\newtheorem{thm}[theorem]{Theorem}

\newtheorem{cor}[theorem]{Corollary}
\newtheorem*{bem*}{Bemerkung}

\newcommand{\ol}{\overline}

\newcommand{\be}{\begin{eqnarray}}
\newcommand{\ee}{\end{eqnarray}}
\newcommand{\ce}{\begin{eqnarray*}}
\newcommand{\de}{\end{eqnarray*}}

\numberwithin{equation}{section}

\begin{document}
\title[Convergence of invariant measures]{Convergence of invariant measures for singular stochastic diffusion equations}
\author[I. Ciotir]{Ioana Ciotir}
\address{Department of Mathematics, Faculty of Economics and Business Administration, ``Al. I. Cuza'' University, Bd. Carol no. 9--11, Ia\c{s}{}i, Romania}
\email{ioana.ciotir@feaa.uaic.ro}
\author[J. M. T\"olle]{Jonas M. T\"olle}
\address{Institut f\"ur Mathematik, Technische Universit\"at Berlin (MA 7-5), Stra\ss{}e des 17. Juni 136, 10623 Berlin, Germany}
\email{toelle@math.tu-berlin.de}
\thanks{Financial support of the DFG Collaborative Research Center 701 (SFB 701) ÒSpectral Structures and Topological Methods in MathematicsÓ (Bielefeld) and the DFG Research Group 718 (Forschergruppe 718) ÒAnalysis and Stochastics in Complex Physical SystemsÓ (Berlin--Leipzig) is gratefully acknowledged.
\\
Both authors would like to thank Viorel Barbu and Michael R\"ockner for helpful comments.
The authors are grateful for the remarks of two referees which helped in improving the paper.}

\begin{abstract}
It is proved that the solutions to the singular stochastic $p$-Laplace equation, $p\in (1,2)$ and the solutions
to the stochastic fast diffusion equation with nonlinearity parameter $r\in (0,1)$ on a bounded open domain $\Lambda\subset\R^d$ with
Dirichlet boundary conditions are continuous in mean, uniformly in time, with respect to the parameters $p$ and $r$ respectively (in the Hilbert spaces $L^2(\Lambda)$, $H^{-1}(\Lambda)$ respectively). The highly
singular limit case $p=1$ is treated with the help of stochastic evolution variational inequalities, where $\mathbbm{P}$-a.s. convergence, uniformly in time, is established.

It is shown that the associated unique invariant measures of the ergodic semigroups converge in the
weak sense (of probability measures).
\end{abstract}
 \keywords{Stochastic evolution equation, stochastic diffusion equation, $p$-Laplace equation, $1$-Laplace equation, total variation flow, fast diffusion equation, ergodic semigroup, unique invariant measure, variational convergence}
 \subjclass[2000]{60H15; 35K67, 37L40, 49J45}

\maketitle

\section{Introduction}

Let $\Lambda\subset\R^d$ be a bounded open domain with Lipschitz boundary $\partial\Lambda$.
Let $\{W(t)\}_{t\ge 0}$ be a $U$-valued cylindrical Wiener process on some filtered probability space
$(\Omega,\Fscr,\{\Fscr(t)\}_{t\ge 0},\P)$, where $U$ is a separable Hilbert space.

We are interested in the following two (families of) stochastic diffusion equations,
the stochastic $p$-Laplacian equation, $p\in (1,\infty)$, $B\in L_2(U,L^2(\Lambda))$,
\[\label{sp}
(\textup{PL}_p)\left\{
\begin{aligned}
\d X_{p}\left( t\right)&=\operatorname{div}\left[\lrabs{\nabla X_{p}(t)}^{p-2}\nabla X_{p}(t)\right]\d t+B\,\d W\left( t\right) \quad&\text{in~}(0,T) \times\Lambda, \\
X_{p}\left( t\right)&=0\quad&\text{on~}\left( 0,T\right) \times \partial \Lambda, \\ 
X_{p}\left( 0\right)&=x\in L^2(\Lambda)\quad&\text{in~}\Lambda.%
\end{aligned}
\right.
\]

The deterministic $p$-Laplace equation arises from geometry,
quasi-regular mappings, fluid dynamics and plasma physics, see \cite{DiBe,Diaz}.
In \cite{Lad}, $(\textup{PL}_p)$ with $B\equiv 0$ is suggested as a model of motion
of non-Newtonian fluids. See \cite{Liu1} for the stochastic equation.

We are also interested in the
stochastic fast diffusion equation $r\in (0,\infty)$, $B\in L_2(U,H^{-1}(\Lambda))$,
\[(\textup{FD}_r)
\left\{\begin{aligned}d Y_r(t) &= \Delta  \left( |Y_r(t)|^{r-1} Y_r(t)   \right)\,dt +B\,d W(t),  &\text{in~}(0,T) \times\Lambda,\\
 Y_r(t)&=0, \  &\text{on~}\left( 0,T\right) \times \partial \Lambda, \\
 Y_r(0)&=y\in H^{-1}(\Lambda),        &\text{in~}\Lambda,\end{aligned}\right.
\]
which models diffusion in plasma physics, curvature flows and
self-organized criticality in sandpile models, see e.g. \cite{BDPR2,BeHo,Rose,Vaz2} and the references therein.

The above equations considered are called \emph{singular} for $p\in (1,2)$, $r\in (0,1)$ and \emph{degenerate}
for $p\in (2,\infty)$, $r\in (1,\infty)$ (porous medium equation). In this paper, we shall investigate the former case.

For $p=1$, equation $(\textup{PL}_1)$ can be heuristically written as a stochastic evolution inclusion,
$B\in L_2(U,H^{-1}(\Lambda))$,
\[\label{sp1}
(\textup{PL}_1)\left\{\begin{aligned}\d X_{1}\left( t\right)&\in\operatorname{div}\left[\operatorname{Sgn}(\nabla X_{1}(t))\right]\d t+B\,\d W\left( t\right) \quad&\text{in~}(0,T) \times\Lambda, \\
X_{1}\left( t\right)&=0\quad&\text{on~}\left( 0,T\right) \times \partial \Lambda, \\ 
X_{1}\left( 0\right)&=x\in L^2(\Lambda)\quad&\text{in~}\Lambda,\end{aligned}\right.
\]
where $\operatorname{Sgn}:\R^d\to 2^{\R^d}$ is defined by
\[\operatorname{Sgn}(u):=\left\{\begin{aligned}&\dfrac{u}{\abs{u}},&&\;\;\text{if}\;\;u\in\R^d\setminus\{0\},\\
                  &\left\{v\in\R^d\;\vert\;\abs{v}\le 1\right\},&&\;\;\text{if}\;\;u=0.
                 \end{aligned}\right.\]
A precise characterization of the $1$-Laplace operator can be found in \cite{ABCM,ACM,Schu}.
A typical $2$-dimensional example for the so-called total variation flow can be
found in image restoration, see \cite{AFP,ACM,AuKo} and the references therein.

We shall, however, take use of the stochastic evolution variational inequality-formulation as in \cite{BDPR}.

We are particularly interested in continuity of the solutions in the parameters $p$ and $r$, especially for
the case $p\to 1$. Stochastic Trotter-type results in this direction
have been obtained by the first named author in \cite{Ciot,Ciot2,Ciot3}.
However, for the case $p\to 1$, we shall need the theory of Mosco convergence of convex functionals as in \cite{A},
since no strong characterization of the limit is available (which could be treated by Yosida-approximation methods). For $B=0$ (i.e., the deterministic equation), the convergence of solutions to
the evolution problem (PL$_p$) was proved in \cite{GKY,Toe3}. See also \cite[Ch. 8.3]{Toe2}.

With the help of a uniqueness result for invariant measures of the equations considered,
obtained by Liu and the second named author \cite{LiuToe}, we prove tightness and the weak
convergence (weak continuity) of invariant measures associated to the ergodic semigroups of the equations
(PL$_p$) and (FD$_r$). See \cite{BDP,BDP2,DPZ2,GesToe} for other result in this direction.

\subsection*{Organization of the paper}

In Section \ref{convergencesec}, we prove that the solutions to the basic examples
are continuous in the parameters $p$ and $r$ resp.

In Section \ref{invsec}, The result of Section \ref{convergencesec} is combined with the uniqueness of invariant
measures proved in \cite{LiuToe} in order to obtain the weak continuity
of invariant measures in the parameters $p$ and $r$ resp.

In Section \ref{singsec}, we prove a convergence result for the stochastic $p$-Laplace equation as $p\to 1$, using another
notion of a solution. For the limit $p=1$, however, uniqueness of the invariant measure
is an open question. The matter is further investigated in \cite{GesToe}.

The Appendix collects some well-known results on Mosco (variational) convergence and
Mosco convergence in $L^p$-spaces, needed for the proof in Section \ref{singsec}.

\section{Convergence of solutions}\label{convergencesec}

Compare with \cite[Theorem 2]{Ciot2}.

\begin{thm}\label{PLthm}
Let $\{p_n\}\subset\left(1\vee \frac{2d}{2+d},2\right]$, $n\in\N$, $p_0\in\left(1\vee \frac{2d}{2+d},2\right]$ such
that $p_n\to p_0$. Let $X_n:=X_{p_n}$, $n\in\N$, $X_0:=X_{p_0}$
be the solutions to $(\textup{PL}_{p_n})$, $n\in\N$, $(\textup{PL}_{p_0})$ resp.
Then for $x\in L^2(\Lambda)$.

\[\lim_n\E\left[\sup_{t\in [0,T]}\norm{X_n(t)-X_0(t)}_{L^2(\Lambda)}^2\right]=0.\]
\end{thm}
\begin{proof}
For $p\in (1,\infty)$, define $a_p:\R^d\to\R^d$ by $a_p(x):=\abs{x}^{p-2}x$.
Furthermore, let
$A_{p}:W_{0}^{1,p}\left(\Lambda\right) \rightarrow
(W_{0}^{1,p})^\ast\left( \Lambda\right)$ be defined by
$A_{p}\left( y\right):=-\operatorname{div}\left[ a_{p}\left( \nabla y\right) %
\right]$, where $y\in W_{0}^{1,p}\left( \Lambda\right)$.
To be more specific,
\[
\gdualdel{(W^{1,p})^\ast}{W^{1,p}}{A_{p}\left( y\right)}{z} =\int_{\Lambda}\lrbr{a_p(\nabla y)}{\nabla z}\,\d\xi ,\quad\forall z\in W_{0}^{1,p}\left( \Lambda\right) .
\]

We first consider the following approximating equations for $(\textup{PL}_p)$
\begin{equation} \label{approx1}
\left\{ 
\begin{aligned}
\d X_{p}^{\varepsilon }\left( t\right) +A_{p}^{\varepsilon }\left(
X_{p}^{\varepsilon }(t)\right) \d t&=B\,\d W\left( t\right) \\ 
X_{p}^{\varepsilon }\left( 0\right) &=x%
\end{aligned}%
\right. 
\end{equation}%
where for any $u\in L^2(\Lambda)$,
\begin{equation*}
A_{p}^{\varepsilon }\left( u\right) =-\left( 1-\varepsilon \Delta \right)
^{-1}\operatorname{div}\left[ a_{p}^{\varepsilon }\left( \nabla \left( 1-\varepsilon
\Delta \right) ^{-1}u\right) \right]
\end{equation*}%
and $a_{p}^{\varepsilon }$ is the Yosida approximation of $a_{p}$ i.e., for any $r\in\R^d$,%
\begin{equation*}
a_{p}^{\varepsilon }\left( r\right) =\frac{1}{\varepsilon }\left( 1-\left(
1+\varepsilon a_{p}\right) ^{-1}\left( r\right) \right).
\end{equation*}%
In particular, for $u,v\in L^2(\Lambda)$,
\[\left(A_p^\eps(u),v\right)_{L^2(\Lambda)}=\int_\Lambda\lrbr{a_p^\eps(\nabla R_\eps u)}{\nabla R_\eps(v)}\,\d\xi,\]
where $R_\eps:=(1-\eps\Delta)^{-1}$ is the resolvent of the Dirichlet Laplacian.

We shall use the following strategy ($\P\text{-a.s.}$)
\begin{multline*}
\left\Vert X_{n}\left( t\right) -X_0\left( t\right) \right\Vert _{L^2(\Lambda)}^2\\
\leq 3\left\Vert X_{n}\left( t\right) -X_{n}^{\varepsilon }\left( t\right)
\right\Vert _{L^2(\Lambda)}^2+3\left\Vert X_{n}^{\varepsilon }\left( t\right)
-X^{\varepsilon }_0\left( t\right) \right\Vert _{L^2(\Lambda)}^2\\
+3\left\Vert X^{\varepsilon
}_0\left( t\right) -X_0\left( t\right) \right\Vert _{L^2(\Lambda)}^2\\
=:I_1(n,\eps)+I_2(n,\eps)+I_3(\eps).
\end{multline*}%
uniformly in $t\in \left[ 0,T\right] .$

At this point we need to prove the following lemma.
We introduce the notation $r^p_\eps \left(
r\right):=\left( 1+\varepsilon a_{p}\right) ^{-1}\left( r\right) $.

\begin{lem}\label{apriorilemma}
Under our assumptions, if we let $X_p^\eps$ be the solution to \eqref{approx1} and
$\tilde{X}_p^\eps:=(1-\eps\Delta)^{-1}X_p^\eps$, we have that%
\begin{equation}
\mathbbm{E}\int_{0}^{t}\int_{\Lambda}\left\vert r_{\varepsilon }^p\left( \nabla 
\tilde{X}_{p}^{\varepsilon }\left( s\right) \right) \right\vert ^{p}\,\d\xi\,
\d s\leq C_{t}\left( \left\Vert x\right\Vert _{L^2(\Lambda)}^{2}+\norm{B}^2_{HS}\right) ,
\label{estimare_J}
\end{equation}%
for all $t\in \left[ 0,T\right]$.
\end{lem}
\begin{proof}
We know by the definition of $a_{p}$ that 
\begin{equation*}
\left\langle a_{p}\left( r\right) ,r\right\rangle \geq\left\vert
r\right\vert ^{p}.
\end{equation*}

On the other hand we have by It\={o}'s formula, applied to the function $u\mapsto\norm{u}_{L^2(\Lambda)}^2$, that 
\begin{equation}\label{ito}
\mathbb{E}\left\Vert X_{p}^{\varepsilon }\left( t\right) \right\Vert
_{L^2(\Lambda)}^{2}+2\mathbb{E}\int_{0}^{t}\int_{\Lambda}\left\langle
a_{p}^{\varepsilon }\left( \nabla \tilde{X}_{p}^{\varepsilon }\left(
s\right) \right) ,\nabla \tilde{X}_{p}^{\varepsilon }\left( s\right)
\right\rangle \,\d\xi\,\d s
\end{equation}%
\begin{equation*}
\leq C_{t}\left( \left\Vert x\right\Vert _{L^2(\Lambda)}^{2}+\norm{B}^2_{HS}\right) .
\end{equation*}

By the definition of the Yosida approximation we have that 
\begin{equation*}
a_{p}^{\varepsilon }\left( r\right) =a_{p}\left( r^p_\eps \left(
r\right) \right) 
\end{equation*}%
and 
\begin{equation*}
\left\langle a_{p}^{\varepsilon }\left( r\right) ,r\right\rangle
=\left\langle a_{p}^{\varepsilon }\left( r^p_\varepsilon \left( r\right)
\right) ,r^p_\varepsilon \left( r\right) \right\rangle +\frac{1}{%
\varepsilon }\left\vert r-r^p_\varepsilon \left( r\right) \right\vert ^{2}.
\end{equation*}

We rewrite as follows%
\begin{eqnarray*}
&&\mathbb{E}\int_{0}^{t}\int_{\Lambda}\left\langle a_{p}^{\varepsilon
}\left( \nabla \tilde{X}_{p}^{\varepsilon }\left( s\right) \right) ,\nabla 
\tilde{X}_{p}^{\varepsilon }\left( s\right) \right\rangle \,\d\xi\,\d s \\
&\geq &\mathbb{E}\int_{0}^{t}\int_{\Lambda}\left\langle a_{p}\left(
r^p_{\varepsilon }\left( \nabla \tilde{X}_{p}^{\varepsilon }\left( s\right)
\right) \right) ,r^p_{\varepsilon }\left( \nabla \tilde{X}_{p}^{\varepsilon
}\left( s\right) \right) \right\rangle \,\d\xi\,\d s \\
&\geq &\mathbb{E}\int_{0}^{t}\int_{\Lambda}\left\vert r^p_{\varepsilon
}\left( \nabla \tilde{X}_{p}^{\varepsilon }\left( s\right) \right)
\right\vert ^{p}\,\d\xi\,\d s.
\end{eqnarray*}%
Plugging into \eqref{ito} proves \eqref{estimare_J}.
\end{proof}

We shall prove now that $\P\text{-a.s.}$
\begin{equation*}
\underset{\varepsilon \rightarrow 0}{\lim }\underset{t\in \left[ 0,T\right] }%
{\sup }\left\Vert X_{p}\left( t\right) -X_{p}^{\varepsilon }\left( t\right)
\right\Vert _{L^2(\Lambda)}^2=0,\quad \text{uniformly in }p\in \left(
\tfrac{2d}{d+2},2\right) .
\end{equation*}

We set $\tilde{X}_{p}^{\varepsilon }=\left( 1-\varepsilon \Delta \right)
^{-1}X_{p}^{\varepsilon }$ and $\tilde{X}_{p}^{\lambda }=\left( 1-\lambda
\Delta \right) ^{-1}X_{p}^{\lambda }$. Then by \eqref{approx1}, we have that 
\begin{multline*}
\frac{1}{2}\left\Vert X_{p}^{\varepsilon }\left( t\right) -X_{p}^{\lambda
}\left( t\right) \right\Vert _{L^2(\Lambda)}^2\\
+\int_{0}^{t}\int_{\Lambda}\left\langle a_{p}^{\varepsilon }\left(
\nabla \tilde{X}_{p}^{\varepsilon }\left( s\right) \right) -a_{p}^{\lambda
}\left( \nabla \tilde{X}_{p}^{\lambda }\left( s\right) \right) ,\nabla 
\tilde{X}_{p}^{\varepsilon }\left( s\right) -\nabla \tilde{X}_{p}^{\lambda
}\left( s\right) \right\rangle \,\d\xi\,\d s=0\quad\P\text{-a.s.}.
\end{multline*}

Setting $\nabla \tilde{X}_{p}^{\varepsilon }\left( s\right) =u^{\varepsilon }
$ and $\nabla \tilde{X}_{p}^{\lambda }\left( s\right) =u^{\lambda }$ and
using 
\begin{equation*}
a_{p}^{\varepsilon }\left( u\right) \in a_{p}\left( \left( 1+\varepsilon
a_{p}\right) ^{-1}\left( u\right) \right) ,
\end{equation*}%
we get by the monotonicity of $a_{p}$ that%
\begin{eqnarray*}
&&\left\langle a_{p}^{\varepsilon }\left( u^{\varepsilon }\right)
-a_{p}^{\lambda }\left( u^{\lambda }\right) ,u^{\varepsilon }-u^{\lambda
}\right\rangle  \\
&\geq &\left\langle a_{p}^{\varepsilon }\left( u^{\varepsilon }\right)
-a_{p}^{\lambda }\left( u^{\lambda }\right) ,\varepsilon a_{p}^{\varepsilon
}\left( u^{\varepsilon }\right) -\lambda a_{p}^{\lambda }\left( u^{\lambda
}\right) \right\rangle .
\end{eqnarray*}

This leads to%
\begin{multline}\label{estim}
\frac{1}{2}\left\Vert X_{p}^{\varepsilon }\left( t\right) -X_{p}^{\lambda
}\left( t\right) \right\Vert _{L^2(\Lambda)}^2\\
\leq \int_{0}^{t}\int_{\Lambda}\left( \varepsilon \left\vert
a_{p}^{\varepsilon }\left( \nabla \tilde{X}_{p}^{\varepsilon }\left(
s\right) \right) \right\vert ^{2}+\lambda \left\vert a_{p}^{\lambda
}\left( \nabla \tilde{X}_{p}^{\lambda }\left( s\right) \right) \right\vert
^{2}\right)\,\d\xi\,\d s\quad\P\text{-a.s.}.
\end{multline}

We can now prove that $\P\text{-a.s.}$

\begin{equation}
\int_{0}^{t}\int_{\Lambda}\left\vert a_{p}^{\varepsilon }\left( \nabla 
\tilde{X}_{p}^{\varepsilon }\left( s\right) \right) \right\vert ^{2}\,\d\xi\,
\d s\leq C_t \label{bound}
\end{equation}%
for some $C_t$ independent of $p$ and $\varepsilon$.

Using Jensen's inequality (for $t\mapsto t^{p/(2p-2)}$) and taking into account that
$\left\vert a_{p}\left( r\right) \right\vert \leq\left\vert
r\right\vert ^{p-1},$ we obtain%
\begin{equation}\label{jenseneq}\begin{split}
&\int_{0}^{t}\int_{\Lambda}\left\vert a_{p}^{\varepsilon }\left(
\nabla \tilde{X}_{p}^{\varepsilon }\left( s\right) \right) \right\vert
^{2}\,\d\xi\,\d s \\
\leq &(t\left\vert \Lambda\right\vert)^{1-((2p-2)/p)}\left(
\int_{0}^{t}\int_{\Lambda}\left\vert a_{p}\left( r^p_{\varepsilon }\left(
\nabla \tilde{X}_{p}^{\varepsilon }\left( s\right) \right) \right)
\right\vert ^{p/(p-1)}\,\d\xi\,\d s\right) ^{(2p-2)/p} \\
\leq &(1+t\left\vert \Lambda\right\vert)\left( \int_{0}^{t}\int_{\Lambda}\left\vert r^p_{\varepsilon
}\left( \nabla \tilde{X}_{p}^{\varepsilon }\left( s\right) \right)
\right\vert ^{p}\,\d\xi\,\d s\right) ^{(2p-2)/p}\\
\leq &C_t+C_t\left( \int_{0}^{t}\int_{\Lambda}\left\vert r^p_{\varepsilon
}\left( \nabla \tilde{X}_{p}^{\varepsilon }\left( s\right) \right)
\right\vert ^{p}\,\d\xi\,\d s\right),
\end{split}\end{equation}%
where $\left\vert \Lambda\right\vert =\int_{\Lambda}\,\d\xi $.

Now by Lemma \ref{apriorilemma} we have \eqref{bound} for a constant $C_t$
independent of $p$ and $\varepsilon ,$ and passing to the limit for $%
\varepsilon ,\lambda \rightarrow 0$ in \eqref{estim} we get that $\P\text{-a.s.}$%
\begin{equation*}
\underset{\varepsilon \rightarrow 0}{\lim }\underset{t\in \left[ 0,T\right] }%
{\sup }\left\Vert X_{p}\left( t\right) -X_{p}^{\varepsilon }\left( t\right)
\right\Vert _{L^2(\Lambda)}^2=0,\quad\text{uniformly in }p\in \left(1\vee
\frac{2d}{d+2},2\right) .
\end{equation*}

As a consequence, $I_1(n,\eps)$ and $I_3(\eps)$ tend to zero as $\eps\downarrow 0$, uniformly in $n$.

For $I_{2}(n,\eps)$, using the monotonicity of $a_{p_{n}}^{\varepsilon }$we have%
\begin{multline*}
\frac{1}{2}\left\Vert X_{p_{n}}^{\varepsilon }\left( t\right)
-X_{p_{0}}^{\varepsilon }\left( t\right) \right\Vert _{L^{2}\left( \Lambda
\right) }^{2}\\
+\int_{0}^{t}\int_{\Lambda }\left\langle a_{p_{n}}^{\varepsilon }\left(
\nabla \tilde{X}_{p_{0}}^{\varepsilon }\left( s\right) \right)
-a_{p_{0}}^{\epsilon }\left( \nabla \tilde{X}_{p_{0}}^{\varepsilon }\left(
s\right) \right) ,\nabla \tilde{X}_{p_{n}}^{\varepsilon }\left( s\right)
-\nabla \tilde{X}_{p_{0}}^{\varepsilon }\left( s\right) \right\rangle
_{d}d\xi ds\leq 0.
\end{multline*}

Since 
\begin{eqnarray*}
&&\frac{1}{2}\left\Vert X_{p_{n}}^{\varepsilon }\left( t\right)
-X_{p_{0}}^{\varepsilon }\left( t\right) \right\Vert _{L^{2}\left( \Lambda
\right) }^{2} \\
&\leq &\int_{0}^{t}\int_{\Lambda }\left[ \left( 1-\varepsilon \Delta \right)
^{-1}\operatorname{div}\left( a_{p_{n}}^{\varepsilon }\left( \nabla \tilde{X}%
_{p_{0}}^{\varepsilon }\left( s\right) \right) -a_{p_{0}}^{\epsilon }\left(
\nabla \tilde{X}_{p_{0}}^{\varepsilon }\left( s\right) \right) \right) %
\right] \left[ X_{p_{n}}^{\varepsilon }\left( s\right)
-X_{p_{0}}^{\varepsilon }\left( s\right) \right] d\xi ds \\
&\leq &\left( \int_{0}^{t}\int_{\Lambda }\left( \left( 1-\varepsilon \Delta
\right) ^{-1}\operatorname{div}a_{p_{n}}^{\varepsilon }\left( \nabla \tilde{X}%
_{p_{0}}^{\varepsilon }\left( s\right) \right) -\left( 1-\varepsilon \Delta
\right) ^{-1}\operatorname{div}a_{p_{0}}^{\epsilon }\left( \nabla \tilde{X}%
_{p_{0}}^{\varepsilon }\left( s\right) \right) \right) ^{2}d\xi ds\right)
^{1/2} \\
&&\quad \quad \quad \times \left( \int_{0}^{t}\int_{\Lambda }\left(
X_{p_{n}}^{\varepsilon }\left( s\right) -X_{p_{0}}^{\varepsilon }\left(
s\right) \right) ^{2}d\xi ds\right) ^{1/2}
\end{eqnarray*}

We only need to prove that

\begin{equation*}
\left( \int_{0}^{t}\int_{\Lambda }\left( \left( 1-\varepsilon \Delta \right)
^{-1}\operatorname{div}a_{p_{n}}^{\varepsilon }\left( \nabla \tilde{X}%
_{p_{0}}^{\varepsilon }\left( s\right) \right) -\left( 1-\varepsilon \Delta
\right) ^{-1}\operatorname{div}a_{p_{0}}^{\epsilon }\left( \nabla \tilde{X}%
_{p_{0}}^{\varepsilon }\left( s\right) \right) \right) ^{2}d\xi ds\right)
^{1/2}\rightarrow 0
\end{equation*}%
and that follows from 
\begin{equation}
A_{p_{n}}^{\varepsilon }\left( u\right) \rightarrow A_{p_{0}}^{\varepsilon
}\left( u\right) ,~\text{\ strongly in \thinspace }L^{2}\left( \left(
0,T\right) \times \Lambda \right) ,  \label{A_p_alpha}
\end{equation}%
where $A_{p_{n}}^{\varepsilon }\left( u\right) =\left( 1-\varepsilon \Delta
\right) ^{-1}\operatorname{div}a_{p_{n}}^{\varepsilon }\left( u\right) $ (as in \eqref{approx1}).

Indeed, we obtain \eqref{A_p_alpha} by the following arguments:

Since $a_{p_{n}}^{\varepsilon }\left( u\right) \rightarrow
a_{p_{0}}^{\epsilon }\left( u\right) $ pointwise,
which follows from Lemma \ref{lem2} and \cite[Proposition 3.29]{A},
and since $\left\{
a_{p_{n}}^{\varepsilon }\left( u\right) \right\} _{n}$ is bounded $a.e.$ on $%
\left( 0,T\right) \times \Lambda $ we get by Lebesgue's dominated
convergence theorem%
\begin{equation*}
\left\langle \operatorname{div}a_{p_{n}}^{\varepsilon }\left( u\right) -\operatorname{div}%
a_{p_{0}}^{\epsilon }\left( u\right) ,v\right\rangle _{L^{2}\left( \left(
0,T\right) \times \Lambda \right) }=\left\langle a_{p_{n}}^{\varepsilon
}\left( u\right) -a_{p_{0}}^{\epsilon }\left( u\right) ,\nabla
v\right\rangle _{L^{2}\left( \left( 0,T\right) \times \Lambda \right) }%
\overset{n}{\rightarrow }0,
\end{equation*}%
for all $v\in L^{2}\left( \left( 0,T\right) \times \Lambda \right) $.

That means 
\begin{equation*}
\operatorname{div}a_{p_{n}}^{\varepsilon }\left( u\right) \rightarrow \operatorname{div}%
a_{p_{0}}^{\epsilon }\left( u\right) ,\text{ weakly in }L^{2}\left( \left(
0,T\right) \times \Lambda \right)
\end{equation*}%
and this leads to%
\begin{equation*}
\left( 1-\varepsilon \Delta \right) ^{-1}\operatorname{div}a_{p_{n}}^{\varepsilon
}\left( u\right) \rightarrow \left( 1-\varepsilon \Delta \right) ^{-1}\operatorname{%
div}a_{p_{0}}^{\varepsilon }\left( u\right) ,\text{ strongly in }L^{2}\left(
\left( 0,T\right) \times \Lambda \right) ,
\end{equation*}%
which is (\ref{A_p_alpha}).

We have proved that
\[\lim_n\sup_{t\in [0,T]}\norm{X_n(t)-X_0(t)}_{L^2(\Lambda)}=0\quad\P\text{-a.s.}\]
The convergence
\[\lim_n\E\left[\sup_{t\in [0,T]}\norm{X_n(t)-X_0(t)}_{L^2(\Lambda)}^2\right]=0\]
is established by Lebesgue's dominated convergence theorem and
\cite[Eq. (1.3)]{Liu1}, where the constant can be controlled uniformly in $p$ by
It\={o}'s formula, Poincar\'e inequality and Gr\"{o}nwall's lemma. We
refer to \cite{Tal} for the $p$-dependence of Poincar\'e constants.
\end{proof}

\begin{thm}
Let $\{r_n\}\subset \left(0\vee\frac{d-2}{d+2},1\right]$, $n\in\N$, $r_0\in \left(0\vee\frac{d-2}{d+2},1\right]$ such
that $r_n\to r_0$. Let $Y_n:=Y_{r_n}$, $n\in\N$, $Y_0:=Y_{r_0}$
be the solutions to $(\textup{FD}_{r_n})$, $n\in\N$, $(\textup{FD}_{r_0})$ resp.
Then for $y\in H^{-1}(\Lambda)$,
\[\lim_n\E\left[\sup_{t\in [0,T]}\norm{Y_n(t)-Y_0(t)}_{H^{-1}(\Lambda)}^2\right]=0.\]
\end{thm}
\begin{proof}
We need to show that 
\begin{equation*}
\underset{n}{\lim }\,\mathbb{E}\left[\underset{t\in \left[ 0,T\right] }{\sup }%
\left\Vert Y_{n}\left( t\right) -Y_{0}\left( t\right) \right\Vert^2
_{H^{-1}\left( \Lambda \right) }\right]= 0.
\end{equation*}

Using the same approximation as in \cite{BDP2} consider 
\begin{eqnarray*}
&&\left\Vert Y_{n}\left( t\right) -Y_{0}\left( t\right) \right\Vert
_{H^{-1}\left( \Lambda \right) }\medskip \\
&\leq &\left\Vert Y_{n}\left( t\right) -Y_{n}^{\varepsilon }\left( t\right)
\right\Vert _{H^{-1}\left( \Lambda \right) }+\left\Vert Y_{n}^{\varepsilon
}\left( t\right) -Y_{0}^{\varepsilon }\left( t\right) \right\Vert
_{H^{-1}\left( \Lambda \right) }+\left\Vert Y_{0}^{\varepsilon }\left(
t\right) -Y_{0}\left( t\right) \right\Vert _{H^{-1}\left( \Lambda \right)
}\medskip \\
&=&I_{1}+I_{2}+I_{3}.
\end{eqnarray*}

For $I_{1}$ and $I_{3}$ we have the convergence uniformly in $r_{n}$ for $%
r_{n}>1/2,$ arguing as in \cite{BDP2}, Proposition 2.6 and using at
the end Jensen's inequality for $L^{2}\left( \Lambda\right) \subset
L^{2r_{n}}\left( \Lambda\right) .$

For $I_{2}$ note that the pointwise convergence of $\Psi _{r_{n}}\left(
x\right) =\left\vert x\right\vert ^{r_{n}-1}x$ to $\Psi _{r_{0}}\left(
x\right) =\left\vert x\right\vert ^{r_{0}-1}x$ imply the convergence of the
resolvent in $\mathbb{R}$ and then we get the result arguing as in \cite%
{Ciot}.
\end{proof}

\section{Convergence of invariant measures}\label{invsec}

In this section, we shall present a result on convergence of invariant measures associated
to equations $(\textup{PL}_{p})$, $(\textup{FD}_{r})$ respectively.

Let $\{X^x_p(t)\}_{t\ge 0}$ be the variational solution associated to equation $(\textup{PL}_{p})$
starting at $x\in L^2(\Lambda)$. Similarly, let $\{Y^y_r(t)\}_{t\ge 0}$ be the variational solution associated to equation $(\textup{FD}_{r})$
starting at $y\in H^{-1}(\Lambda)$.

Let
\[P_t^{p}\,F(x):=\E\left[F(X_{p}^x(t))\right],\quad F\in C_b(L^2(\Lambda)),\;\;t\ge 0,\]
be the semigroup
associated to equation $(\textup{PL}_{p})$.

Let
\[Q_t^{r}\,G(y):=\E\left[G(Y_{r}^y(t))\right],\quad G\in C_b(H^{-1}(\Lambda)),\;\;t\ge 0,\]
be the semigroup
associated to equation $(\textup{FD}_{r})$.

Recently, Liu and the second named author obtained the following result:

\begin{prop}\label{LTprop}
Suppose that
$p\in\left(1\vee \frac{2d}{2+d},2\right]$,
$r\in \left(0\vee\frac{d-2}{d+2},1\right]$.
Then $\{P_t^p\}$ and $\{Q_t^r\}$ are ergodic and admit unique invariant measures $\mu_p$, $\nu_r$
respectively. It holds that $\mu_p$ is supported by $W^{1,p}_0(\Lambda)$ and $\nu_r$ is supported by $L^{r+1}(\Lambda)$. Also
 \begin{equation}\label{mubound} \int_{L^2(\Lambda)}\|x\|_{1,p}^p \,\mu_p(d x)<+\infty,\end{equation}
and
 \begin{equation}\label{mubound2} \int_{H^{-1}(\Lambda)}\|y\|_{r+1}^{r+1} \,\nu_r(d y)<+\infty.\end{equation}
\end{prop}
\begin{proof}
See \cite[Propositions 3.2 and 3.4]{LiuToe}.
\end{proof}

\begin{thm}\label{measureconvergence}
\begin{enumerate}[(i)]
\item
Let $\{p_n\}\subset\left(1\vee \frac{2d}{2+d},2\right]$, $n\in\N$, $p_0\in\left(1\vee \frac{2d}{2+d},2\right]$ such
that $p_n\to p_0$. Set $P_t^n:=P_t^{p_n}$, $P_t^0:=P_t^{p_0}$.

Then the unique invariant measures $\mu_n$, $n\in\N$, $\mu_0$ resp. associated to $\{P_t^n\}$, $n\in\N$, $\{P_t^0\}$ converge in the weak sense, i.e.
\[\lim_n\int_{L^2(\Lambda)}F(x)\,\mu_{n}(dx)=\int_{L^2(\Lambda)}F(x)\,\mu_{0}(dx)\quad\forall F\in C_b(L^2(\Lambda)).\]

\item
Let $\{r_n\}\subset \left(0\vee\frac{d-2}{d+2},1\right]$, $n\in\N$, $r_0\in \left(0\vee\frac{d-2}{d+2},1\right]$ such
that $r_n\to r_0$. Set $Q_t^n:=Q_t^{r_n}$, $Q_t^0:=Q_t^{r_0}$.

Then the unique invariant measures $\nu_n$, $n\in\N$, $\nu_0$ resp. associated to $\{Q_t^n\}$, $n\in\N$, $\{Q_t^0\}$ converge in the weak sense, i.e.
\[\lim_n\int_{L^2(\Lambda)}F(x)\,\nu_{n}(dx)=\int_{L^2(\Lambda)}F(x)\,\nu_{0}(dx)\quad\forall F\in C_b(L^2(\Lambda)).\]
\end{enumerate}
\end{thm}
\begin{proof}
Let us prove (i) first.
By Proposition \ref{LTprop},
we see that $\{P_t^n\}$, $n\in\N$, $\{P_t^0\}$ admit unique invariant measures
$\mu_n$, $n\in\N$, $\mu_0$ resp. Let $p_1:=\inf_n p_n$. By the convergence $p_n\to p_0$,
$p_1\in\left(1\vee \frac{2d}{2+d},2\right]$ and the embedding $W^{1,p_1}_0(\Lambda)\subset L^2(\Lambda)$ is compact.

Let $\theta>0$. Set
\[K_\theta:=\left\{x\in L^2(\Lambda)\;\bigg\vert\;\norm{x}_{1,p_1}^{p_1}\le\theta^{-1}+\abs{\Lambda}\right\}.\]
Clearly, $K_\theta$ is compact in $L^2(\Lambda)$. Now by \eqref{mubound},
\begin{multline*}\mu_n\{K_\theta^{\textup{c}}\}=\mu_n\left\{\norm{\cdot}_{1,p_1}^{p_1}-\abs{\Lambda}\ge\theta^{-1}\right\}
\le\theta\int_{L^2(\Lambda)}\norm{x}_{1,p_n}^{p_n}\,\mu_n(dx)
\le \theta\norm{B}^2_{HS}.\end{multline*}

Hence the family of measures $\{\mu_n\}_{n\in\N}$ is tight and has a weak
accumulation point $\tilde{\mu}$, i.e. $\mu_{n_k}\to\tilde{\mu}$ weakly. By the Krylov--Bogoliubov theorem, for $F\in C_b(L^2(\Lambda))$,
\begin{align*}
\int_{L^2(\Lambda)}F(x)\,\mu_{n_k}(dx)
=&\lim_{T\to+\infty}\frac{1}{T}\int_0^T P_t^{n_k}F(x)\,dt\\
=&\lim_{T\to+\infty}\frac{1}{T}\int_0^T(P_t^{n_k}F(x)-P_t^{0}F(x))\,dt\\
&+\lim_{T\to+\infty}\frac{1}{T}\int_0^T P_t^{0}F(x)\,dt\\
=:&\eps_k+\int_{L^2(\Lambda)}F(x)\,\mu_0(dx)
\end{align*}
By Theorem \ref{PLthm} and dominated convergence, $\eps_k\to 0$ as $k\to+\infty$ and hence
\[\int_{L^2(\Lambda)}F(x)\,\tilde{\mu}(dx)=\int_{L^2(\Lambda)}F(x)\,\mu_0(dx).\]
As a consequence, for the whole sequence, $\mu_n\to\mu_0$ weakly.

The proof for (ii) can be carried out by similar arguments.
\end{proof}

\section{The case $p=1$}\label{singsec}

For $p=1$, the situation is more complicated. We would like to find a convex functional $\Phi^1$ such that the stochastic $1$-Laplace equation
$$\label{1sp}
(\textup{PL}_1)\left\{
\begin{aligned}
\d X_{1}\left( t\right)&=\operatorname{div}\left[\dfrac{\nabla X_{1}(t)}{\lrabs{\nabla X_{1}(t)}}\right]\d t+B\,\d W\left( t\right) \quad&\text{in~}(0,T) \times\Lambda, \\
X_{1}\left( t\right)&=0\quad&\text{on~}\left( 0,T\right) \times \partial \Lambda, \\ 
X_{1}\left( 0\right)&=x\quad&\text{in~}\Lambda,%
\end{aligned}
\right.$$
can be written as
\begin{equation}\label{equ4}
\left\{
\begin{aligned}
\d X_1\left( t\right)&\in -\partial\Phi^1(X_1(t))\,\d t+B\,\d W\left( t\right) \quad&\text{in~}(0,T), \\
X_1\left( 0\right)&=x,\quad&%
\end{aligned}
\right.
\end{equation}%
where $\partial\Phi^1$ is the subdifferential of $\Phi^1$.

We shall need the spaces $BV(\Lambda)$ and $BV(\R^d)$. For $f\in L^1_\loc(\Lambda)$, define the \emph{total variation}
\[
\left\Vert\D f\right\Vert(\Lambda) =\sup \left\{ \int_{\Lambda}f\operatorname{div}\psi\,
\d\xi\;\Big\vert\;\psi \in C_{0}^{\infty }\left(\Lambda;\mathbb{R}^{d}\right),\;\left\vert \psi \right\vert \leq 1\right\}
\]%
$BV(\Lambda)$ is defined to be equal to $\{f\in L^1(\Lambda)\,\vert\,\norm{\D f}(\Lambda)<\infty\}$.
Denote the $d-1$-dimensional Hausdorff measure on $\partial\Lambda$ by $\Hscr^{d-1}$.
For $f\in BV(\Lambda)$ there is an element $f^\Lambda\in L^1(\partial\Lambda,\d\Hscr^{d-1})$ called
the \emph{trace} such that
\[\int_\Lambda f\operatorname{div}\psi\,\d\xi=-\int_\Lambda\lrbr{\psi}{\d[\D f]}+\int_{\partial\Lambda}\lrbr{\psi}{\nu} f^\Lambda\,\d\Hscr^{d-1}
 \quad\forall\psi\in C^1(\ol{\Lambda};\R^d),
\]
where $[\D f]$ denotes the distributional gradient of $f$ on $\Lambda$ (which is a $\R^d$-valued Radon measure here)
and $\nu$ denotes the outer unit normal on $\partial\Lambda$.
$BV(\R^d)$ is defined similarly by setting $\Lambda=\R^d$. Define also $\norm{\D f}(\R^d)$
in the above manner. Note that for $f\in BV(\Lambda)$ (extended by zero outside $\Lambda$) it holds
that $f\in BV(\R^d)$ and that
\begin{equation}\label{eq:Rdtrace}\norm{\D f}(\R^d)=\norm{\D f}(\Lambda)+\int_{\partial\Lambda}\lrabs{f^\Lambda}\,d\Hscr^{d-1},\end{equation}
cf. \cite[Theorem 3.87]{AFP}.
\begin{rem}\label{rem:contcomp}
By Ambrosio et al. \cite[Corollary 3.49]{AFP}, if $d\in\{1,2\}$, then
\[W^{1,1}_0(\Lambda)\subset BV(\Lambda)\subset L^2(\Lambda)\]
continuously. If $d=1$, then
\[BV(\Lambda)\subset\subset L^2(\Lambda)\]
compactly.
\end{rem}
For further results in spaces of functions of bounded variation, we refer to \cite[Ch. 3]{AFP}.

We shall return to equation \eqref{equ4}.
Recall that the subdifferential $\partial\Phi^1$ in $L^2(\Lambda)$ is defined by $\eta\in\partial\Phi^1(x)$ iff
\begin{equation}
\Phi^1 \left( x\right) -\Phi^1 \left( y\right) \leq \int_{\Lambda%
}\eta \left( x-y\right)\,\d\xi ,\quad\forall y\in\dom\Phi^1.
\end{equation}

One possible choice for $\Phi^1$ is the (homogeneous) energy
\[\tilde{\Phi}(u):=\begin{cases}\int_\Lambda\abs{\nabla u}\,\d\xi,\quad&\text{if~}u\in W^{1,1}_0(\Lambda),\\
            +\infty,\quad&\text{if~}u\in L^2(\Lambda)\setminus W^{1,1}_0(\Lambda).
           \end{cases}
\]
In this case, if $u\in W^{1,1}_0(\Lambda)$, and if
$U:=-\operatorname{div}(\sgn(\nabla u))\subset L^2(\Lambda)$, then we have that $u\in\dom\partial\tilde{\Phi}$ and
$U=\partial\tilde{\Phi}(u)$.

However, $\tilde{\Phi}$ fails to be lower semi-continuous in $L^2(\Lambda)$ which is a necessary ingredient
for the theory. Therefore, it is convenient to consider its relaxed functional in $L^2(\Lambda)$, which is equal to
\[\Phi^1(u):=\begin{cases}\norm{\D u}(\R^d),\quad&\text{if~}u\in BV(\Lambda),\\
            +\infty,\quad&\text{if~}u\in L^2(\Lambda)\setminus BV(\Lambda),
           \end{cases}
\]
see equation \eqref{eq:Rdtrace} above.
$\Phi^1$ is proper, convex and lower semi-continuous in $L^2(\Lambda)$ and an extension of
$\tilde{\Phi}$ in the sense that $\dom\Phi^1\supset\dom\tilde{\Phi}$ and $\Phi^1\le\tilde{\Phi}$.
Compare with \cite{ACM,KaSchu,Schu,Toe3}.

Following the approach of Barbu, Da Prato and R\"ockner \cite{BDPR}, we shall give the definition of a solution
for equations $(\textup{PL}_p)$, $p\in[1,2]$.

\begin{defi}\label{soldef}
Set $V_p:=W^{1,p}_0(\Lambda)$, $p\in (1,2]$, $V_1:=BV(\Lambda)$.
Let $\Phi^1$ be defined as above. For $p\in (1,2]$, let
\[\Phi^p(x):=\begin{cases}\dfrac{1}{p}\displaystyle\int_\Lambda\abs{\nabla x}^p\,\d\xi,\quad&\text{if~}u\in W^{1,p}_0(\Lambda),\\
            +\infty,\quad&\text{if~}u\in L^2(\Lambda)\setminus W^{1,p}_0(\Lambda).
           \end{cases}\]

A stochastic process $X=X^x $ with $\mathbbm{P}$-a.s.
continuous sample paths in $H:=L^{2}\left( \Lambda\right) $ is said to be
a solution to equation $(\textup{{PL}}_p)$, $p\in[1,2]$ if%
\[
X\in C_{W}\left( \left[ 0,T\right] ;H\right) \cap L^{p}\left( \left(
0,T\right) \times \Omega ,V_p \right) ,\quad
X\left( 0\right) =x\in H
\]%
and%
\begin{eqnarray*}
&&\frac{1}{2}\left\Vert X\left( t\right) -Y\left( t\right) \right\Vert
^{2}_{L^2(\Lambda)}+\int_{0}^{t}\left( \Phi^p \left( X\left( s\right) \right) -\Phi^p \left(
Y\left( s\right) \right) \right)\,\d s \\
&\leq &\frac{1}{2}\left\Vert x-Y\left( 0\right) \right\Vert
^{2}_{L^2(\Lambda)}+\int_{0}^{t}\left( G\left( s\right) ,X\left( s\right) -Y\left( s\right)
\right)_{L^2(\Lambda)}\,\d s,\quad t\in \left[ 0,T\right] ,
\end{eqnarray*}%
for all $G \in L_{W}^{2}\left( 0,T;H\right) $ and $Y\in
C_{W}\left( \left[ 0,T\right] ;H\right) \cap L^{p}\left( \left( 0,T\right)
\times \Omega ;V_p \right) $ satisfying the
equation 
\begin{equation}\label{Gequ}
\d Y\left( t\right) +G\left( t\right)\,\d t=B\,\d W\left( t\right) ,\quad t\in %
\left[ 0,T\right] .
\end{equation}
\end{defi}

Suppose for a while that $1<p<2$, $d=1,2$.
Arguing as in \cite[Example 4.1.9, Theorem 4.2.4]{PrRoe}, we can easily prove
existence and uniqueness of the solution $X_{p}$ for equation $(\text{PL}_p)$,
in the usual (strong) variational sense, as in Pardoux, Krylov, Rozovski\u{\i} \cite{KrRo0,Pard}. 
We shall refer to
Pr\'ev\^ot, R\"ockner \cite[Definition 4.2.1]{PrRoe}.
By It\={o}'s formula, we
see that $X_p$ is also a solution in the sense of the definition above.

Here, $W(t)$ is a cylindrical Wiener process on $L^2(\Lambda)$ of the form
\[W(t)=\sum_{n=1}^\infty\gamma_n(t)e_n,\quad t\ge 0,\]
where $\{\gamma_n\}$ is a sequence of mutually independent real Brownian motions on a filtered probability space
$(\Omega,\Fscr,\{\Fscr_t\}_{t\ge 0},\mathbbm{P})$ and $\{e_n\}$ is an orthonormal basis of $L^2(\Lambda)$.
We shall make further specifications. $BB^\ast$ is assumed to be a linear, continuous, non-negative, symmetric operator on $L^2(\Lambda)$
with eigenbasis $\{e_n\}$ and corresponding sequence of eigenvalues $\{\lambda_n\}$.
Let $(-\Delta,\dom(-\Delta))$ be the Dirichlet Laplacian in $L^2(\Lambda)$, in particular,
$\dom(-\Delta)=H^2(\Lambda)\cap H_0^1(\Lambda)$. Assume for simplicity that $\{e_n\}$ is an eigenbasis of
$-\Delta$ with corresponding sequence of eigenvalues $\{\mu_n\}$. We shall assume that
\begin{equation}
\sum_{n=1}^\infty\lambda_n^{1+\kappa}\mu_n<\infty
\end{equation}
for some $\kappa>0$. For the situation considered in this paper, it is enough to set $Q:=(-\Delta)^{-1-\delta}$
with $\delta>\frac{1}{2}+\kappa$ for $d=1$ and $\delta>1+\kappa$ for $d=2$.

Regarding equation $(\text{PL}_1)$, well-posedness of the problem as well as existence
and uniqueness of the solution were proved by Barbu, Da Prato and R\"ockner in \cite{BDPR}.

\begin{rem}\label{rem:BV}
Note that in \cite{BDPR}, the space $BV_0(\Lambda)$ is introduced, consisting of $BV(\Lambda)$-functions with zero trace. They claim, however, that the energy
\[\Psi(u):=\begin{cases}\norm{\D u}(\Lambda),\quad&\text{if~}u\in BV_0(\Lambda),\\
            +\infty,\quad&\text{if~}u\in L^2(\Lambda)\setminus BV_0(\Lambda).
           \end{cases}
\]
is lower semi-continuous which is not the case. Consider, for example, a sequence $u_n$ of trace zero Lipschitz functions on $\Lambda$ with $\norm{\D u_n}(\Lambda)=1$ converging in $L^2(\Lambda)$ to $\mathbbm{1}_\Lambda$. Then
\[\varliminf_n\Psi(u_n)=1< +\infty=\Psi(\mathbbm{1}_\Lambda).\]
Fortunately, all results of \cite{BDPR} remain true, if one replaces $\Psi$ (denoted by $\Phi$ in their paper) by $\Phi^1$. We do not repeat the steps taken in the proof of \cite{BDPR} here,
but note that for their existence and uniqueness result relies on an approximation $\{\Psi^\eps\}$ of $\Psi$ which
``does not see'' the trace-term in \eqref{eq:Rdtrace}, i.e. maps $L^2(\Lambda)$ functions on a joint subspace
of $BV_0(\Lambda)$ and $\dom(\Phi^1)$. In fact, $\{\Psi^\eps\}$ is defined similarly to \eqref{eq:approxdefi}.
\end{rem}

Other results of stochastic evolution variational inequalities can be found in
\cite{BDP3,BenRas,Ras2,Ras3,Ras}.

We are now able to formulate the main result of this section.

\begin{thm}\label{mainresult}
Let $d\in\{1,2\}$.
The sequence of solutions $\left\{ X_{p}\right\} _{p}$ to equations $(\textup{\text{PL}}_p)$
is convergent for $p\rightarrow 1$ to the solution $X_1$ of equation
$(\textup{\text{PL}}_1)$, strongly in $L^{2}\left( \Lambda\right)$, uniformly on
$\left[ 0,T\right]$, $\mathbbm{P}-a.s.$, i.e.,%
\[
\lim_{p\to 1}\sup_{t\in [0,T]}
\left\Vert X_{p}\left( t\right) -X_1\left( t\right) \right\Vert _{L^2(\Lambda)}=0,\quad 
\mathbb{P}-a.s.
\]
\end{thm}

There is some evidence that the following conjecture is true,
see \cite{ESvR,GesToe,KPS}.

\begin{conj}\label{1c}
Let $d\in\{1,2\}$.
Then the semigroup
\[P_t^1 F(x):=\mathbbm{E}\left[F\left(X_1(t,x)\right)\right],\quad F\in C_b(L^2(\Lambda)),\]
admits a unique invariant measure $\mu_1$.
\end{conj}

\begin{thm}\label{invthm1}
Let $d=1$. Suppose that Conjecture \ref{1c} is true.
Let $X_p=X_p(t,x)$ be the solution to equation $(\textup{\text{PL}}_p)$, $p\in [1,2]$.
Let $\{p_n\}\subset (1,2]$ such that $\lim_n p_n=1$. Let
\[P_t^p F(x):=\mathbbm{E}\left[F\left(X_p^x(t)\right)\right],\quad\phi\in C_b(L^2(\Lambda)),\]
be the semigroup
associated to equation $(\textup{\text{PL}}_p)$. Let
$\mu_{p_n}$, $n\in\N$, $\mu_{1}$ be the associated unique invariant measures on $L^2(\Lambda)$. Then
\[\mu_{p_n}\to\mu_{1}\quad\text{in the weak sense}.\]
\end{thm}
\begin{proof}
Note that by Remark \ref{rem:contcomp}, the embedding $BV(\Lambda)\subset L^2(\Lambda)$ is compact.
The proof is similar to that of Theorem \ref{measureconvergence}, $W^{1,p_1}_0(\Lambda)$ therein
replaced by $BV(\Lambda)$.
\end{proof}

\begin{proof}[Proof of Theorem \ref{mainresult}]

For each $\eps>0$, let $R_\eps:=(1-\eps\Delta)^{-1}$ be the resolvent of the (Dirichlet) Laplace operator $(-\Delta,\dom(-\Delta))$,
where $\dom(-\Delta)=H^{1}_0(\Lambda)\cap H^{2}(\Lambda)$.
For $p\in [1,2]$, $\eps>0$, let
\begin{equation}\label{eq:approxdefi}\Phi_\eps^p(u):=\int_\Lambda j_\eps^p(\nabla R_\eps u)\,\d\xi,\quad u\in L^2(\Lambda).\end{equation}

\begin{lem}\label{philem}
Let $\{p_n\}\subset[1,2]$ such that $\lim_n p_n=1$. Let $\eps>0$.
Then for $u\in L^2(\Lambda)$, we have that
\begin{equation}\label{regconveq3}
\lim_n\Phi_\eps^{p_n}(u)=\Phi_\eps^1(u).
\end{equation}
Furthermore, if $u_n\rightharpoonup u$ converges weakly in $L^2(\Lambda)$,
we have that
\begin{equation}\label{regconveq4}
\liminf_n\Phi_\eps^{p_n}(u_n)\ge\Phi_\eps^1(u).
\end{equation}\label{lsceq}
Also, each $\Phi_\eps^p$, $p\in [1,2]$, $\eps>0$, is continuous w.r.t. the weak topology of $L^2(\Lambda)$.
\end{lem}
\begin{proof}
Since $R_\eps$ maps to $\dom(-\Delta)\subset H^{1}_0(\Lambda)$, it is clear that
$\nabla R_\eps u\in L^2(\Lambda;\R^d)$ and hence \eqref{regconveq3} follows from \eqref{regconveq1}.

Let $u_n\in L^2(\Lambda)$, $n\in\N$, $u\in L^2(\Lambda)$, such that $u_n\rightharpoonup u$ weakly in $L^2(\Lambda)$.
If we can proof that $\nabla R_\eps u_n\rightharpoonup\nabla R_\eps u$ weakly in $L^2(\Lambda;\R^d)$, we
can apply \eqref{regconveq2} and \eqref{lsceq} follows. Indeed, we even have that $\nabla R_\eps u_n\to\nabla R_\eps u$
strongly in $L^2(\Lambda;\R^d)$.

The last part follows by repeating the compactness argument above and the strong $L^2(\Lambda;\R^d)$-continuity of the $\Psi^p_\eps$'s.
\end{proof}

We first consider the following approximating equations for $(\text{PL}_p)$
\begin{equation} \label{approx3}
\left\{ 
\begin{aligned}
\d X_{p}^{\varepsilon }\left( t\right) +A_{p}^{\varepsilon }\left(
X_{p}^{\varepsilon }\right) \d t&=B\,\d W\left( t\right) \\ 
X_{p}^{\varepsilon }\left( 0\right) &=x%
\end{aligned}%
\right. 
\end{equation}%
where for any $u\in L^2(\Lambda)$,
\begin{equation*}
A_{p}^{\varepsilon }\left( u\right) =-\left( 1-\varepsilon \Delta \right)
^{-1}\operatorname{div}\left[ a_{p}^{\varepsilon }\left( \nabla \left( 1-\varepsilon
\Delta \right) ^{-1}u\right) \right]
\end{equation*}%
and $a_{p}^{\varepsilon }$ is the Yosida approximation of $a_{p}$ i.e., for any $r\in\R^d$,%
\begin{equation*}
a_{p}^{\varepsilon }\left( r\right) =\frac{1}{\varepsilon }\left( 1-\left(
1+\varepsilon a_{p}\right) ^{-1}\left( r\right) \right).
\end{equation*}%
In particular, for $u,v\in L^2(\Lambda)$,
\[\left(A_p^\eps(u),v\right)_{L^2(\Lambda)}=\int_\Lambda\lrbr{a_p^\eps(\nabla R_\eps u)}{\nabla R_\eps(v)}\,\d\xi.\]

We shall consider a similar approximation for equation $(\text{PL}_1)$ 
\begin{equation}
\left\{ 
\begin{aligned}
\d X^{\varepsilon }_1\left( t\right) +A^{\varepsilon }\left( X^{\varepsilon
}_1\right) \d t&=B\,\d W\left( t\right) \\ 
X^{\varepsilon }_1\left( 0\right)&=x%
\end{aligned}%
\right.  \label{approx2}
\end{equation}%
where for any $u\in L^2(\Lambda)$,
\begin{equation*}
A^{\varepsilon }\left( u\right) =-\left( 1-\varepsilon \Delta \right) ^{-1}%
\operatorname{div}\left[ \beta ^{\varepsilon }\left( \nabla \left( 1-\varepsilon
\Delta \right) ^{-1}u\right) \right].
\end{equation*}%
with%
\begin{equation*}
\beta ^{\varepsilon }\left( r\right) =\left\{ 
\begin{aligned}
\dfrac{r}{\varepsilon },&\text{ if }\left\vert r\right\vert \leq
\varepsilon , \\ 
\dfrac{r}{\left\vert r\right\vert },&\text{ if }\left\vert r\right\vert
>\varepsilon .%
\end{aligned}%
\right.
\end{equation*}%
In particular, for $u,v\in L^2(\Lambda)$,
\[\left(A^\eps(u),v\right)_{L^2(\Lambda)}=\int_\Lambda\lrbr{\beta^\eps(\nabla R_\eps u)}{\nabla R_\eps(v)}\,\d\xi.\]

Note that $\beta ^{\varepsilon }$ is the Yosida approximation of the sign function,
i.e., for any $r\in\R^d$,
\begin{equation*}
\beta ^{\varepsilon }\left( r\right) =\frac{1}{\varepsilon }\left( 1-\left(
1+\varepsilon ~\operatorname{sgn}\right) ^{-1}\left( r\right) \right).
\end{equation*}
In particular, $\beta^\eps=\nabla j^\eps$, where $j_\eps$ is the convex function defined by

\begin{equation*}
j^{\varepsilon }\left( r\right) =\left\{ 
\begin{aligned}
\dfrac{\left\vert r\right\vert ^{2}}{2\varepsilon },&\text{ if }%
\left\vert r\right\vert \leq \varepsilon , \\ 
\left\vert r\right\vert -\dfrac{\varepsilon }{2},&\text{ if }\left\vert
r\right\vert >\varepsilon .%
\end{aligned}%
\right.
\end{equation*}

We shall use the following strategy to prove the main result%
\begin{multline*}
\left\Vert X_{p}\left( t\right) -X_1\left( t\right) \right\Vert _{L^2(\Lambda)}\\
\leq \left\Vert X_{p}\left( t\right) -X_{p}^{\varepsilon }\left( t\right)
\right\Vert _{L^2(\Lambda)}+\left\Vert X_{p}^{\varepsilon }\left( t\right)
-X^{\varepsilon }_1\left( t\right) \right\Vert _{L^2(\Lambda)}+\left\Vert X^{\varepsilon
}_1\left( t\right) -X_1\left( t\right) \right\Vert _{L^2(\Lambda)}
\end{multline*}%
$\mathbb{P}$-a.s. and uniformly in $t\in \left[ 0,T\right] .$

\textbf{Step I}

We note that, taking Remark \ref{rem:BV} into account, the result of
\cite[equation (4.8)]{BDPR} remains valid in our case. Hence,
\begin{equation*}
\underset{\varepsilon \rightarrow 0}{\lim }\underset{t\in \left[ 0,T\right] }%
{\sup }\left\Vert X^{\varepsilon }_1\left( t\right) -X_1\left( t\right)
\right\Vert _{L^2(\Lambda)}=0,\quad \mathbb{P}\text{-a.s.}
\end{equation*}

\textbf{Step II}

Note that we have proved above (proof of Theorem \ref{PLthm}) that
\begin{equation*}
\underset{\varepsilon \rightarrow 0}{\lim }\underset{t\in \left[ 0,T\right] }%
{\sup }\left\Vert X_{p}\left( t\right) -X_{p}^{\varepsilon }\left( t\right)
\right\Vert _{L^2(\Lambda)}=0,\quad \mathbb{P}\text{-a.s. uniformly in }p\in \left(
1,2\right) .
\end{equation*}

\textbf{Step III}

In order to complete the proof we still need to show that for all $%
\varepsilon>0 $ fixed we have 
\begin{equation*}
\underset{p\rightarrow 1}{\lim }\underset{t\in \left[ 0,T\right] }{\sup }%
\left\Vert X_{p}^{\varepsilon }\left( t\right) -X^{\varepsilon }_1\left(
t\right) \right\Vert _{L^2(\Lambda)}=0,\quad \mathbb{P}\text{-a.s.}
\end{equation*}%
To this aim, we consider the definition of the solution for equations%
\begin{equation*}
\left\{ 
\begin{aligned}
\d X_{p}^{\varepsilon }\left( t\right) +A_{p}^{\varepsilon }\left(
X_{p}^{\varepsilon }\right) \d t&=B\,\d W\left( t\right)  \\ 
X_{p}^{\varepsilon }\left( 0\right) &=x%
\end{aligned}%
\right. 
\end{equation*}%
as%
\begin{eqnarray*}
&&\frac{1}{2}\left\Vert X_{p}^{\varepsilon }\left( t\right) -Y\left(
t\right) \right\Vert _{L^2(\Lambda)}^{2}+\int_{0}^{t}\left( \Phi^{p}_{\varepsilon
}\left( X_{p}^{\varepsilon }\left( s\right) \right) -\Phi^{p}_{\varepsilon
}\left( Y\left( s\right) \right) \right)\d s \\
&\leq &\frac{1}{2}\left\Vert x-Y\left( 0\right) \right\Vert
_{L^2(\Lambda)}^{2}+\int_{0}^{t}\left( G\left( s\right) ,X_{p}^{\varepsilon }\left(
s\right) -Y\left( s\right) \right) _{L^2(\Lambda)}\d s,\\
&&\qquad\text{\ for all }t\in \left[
0,T\right] ,\text{ }\mathbb{P}\text{-a.s.}
\end{eqnarray*}

We take $Y=X^{\varepsilon }_1$, the solution of equation 
\begin{equation*}
\left\{ 
\begin{aligned}
\d X^{\varepsilon }_1\left( t\right) +A^{\varepsilon }\left( X^{\varepsilon
}_1\right) \d t&=B\,\d W\left( t\right)  \\ 
X^{\varepsilon }_1\left( 0\right)&=x.%
\end{aligned}%
\right. 
\end{equation*}%
and using the definition of the subdifferential we get that%
\begin{equation}\label{soleq}\begin{split}
&\frac{1}{2}\left\Vert X_{p}^{\varepsilon }\left( t\right) -X^{\varepsilon
}_1\left( t\right) \right\Vert _{L^2(\Lambda)}^{2} \\
&\quad \quad \quad \quad +\int_{0}^{t}\left( \Phi^{p}_{\varepsilon }\left(
X_{p}^{\varepsilon }\left( s\right) \right) -\Phi^{p}_{\varepsilon }\left(
X^{\varepsilon }_1\left( s\right) \right) +\Phi^1_{\varepsilon }\left(
X^{\varepsilon }_1\left( s\right) \right) -\Phi^1_{\varepsilon }\left(
X^{\varepsilon }_p\left( s\right) \right) \right)\, \d s \\
\leq &\frac{1}{2}\left\Vert x-X^{\varepsilon }_1\left( 0\right) \right\Vert
_{L^2(\Lambda)}^{2}=0,
\end{split}\end{equation}
for $t\in \left[ 0,T\right]$ and $\mathbb{P}\text{-a.s.}$ By estimate \eqref{ito}, we can
extract a subsequence $\{p_n\}$ with $\lim_n p_n=1$ such that for $X_n^\eps:=X_{p_n}^\eps$
we have that for $\d t$-a.a. $t\in [0,T]$, $X_n^\eps(t)\rightharpoonup Z^\eps(t)$ weakly in $L^2(\Lambda)$ $\mathbbm{P}$-a.s.
for some $\d t\otimes\mathbbm{P}$-measurable $Z^\eps$ that satisfies
\[\sup_{t\in [0,T]}\norm{Z^\eps(t)}_{L^2(\Lambda)}\le\liminf_n\sup_{t\in [0,T]}\norm{X_n(t)}_{L^2(\Lambda)}\quad\mathbbm{P}\text{-a.s.}\]

We shall need following lemma. Set $\Phi^n_\eps:=\Phi^{p_n}_\eps$.
\begin{lem}\label{fatoulem}
\[\Phi^n_\eps(X^\eps_1(\cdot))-\Phi^n_\eps(X_n^\eps(\cdot))+\Phi^1_\eps(X_n^\eps(\cdot))-\Phi^1_\eps(X^\eps_1(\cdot))\]
is $\mathbbm{P}$-a.s. bounded above by a function in $L^\infty(0,T)$.
\end{lem}
\begin{proof}
Set $u:=X_n^\eps(\cdot)$, $v:=X^\eps_1(\cdot)$. Recall that in our notation, $R_\eps:=(1-\eps\Delta)^{-1}$.

Let us treat the term $\Phi^1_\eps(u)-\Phi^1_\eps(v)$ first. By the definition of the subgradient
it is bounded by $(\nabla\Phi^1_\eps(u),u-v)_{L^2(\Lambda)}$. But this term is equal
to
\[\int_\Lambda\lrbr{\beta^\eps(\nabla R_\eps(u))}{\nabla R_\eps(u-v)}\,\d\xi.\]
Since $\abs{\beta^\eps}\le 1$, we
get that the latter is bounded by $\norm{\nabla R_\eps(u-v)}_{L^2(\Lambda;\R^d)}$. By the proof of
Lemma \ref{philem}, $\nabla R_\eps$ is a bounded operator from $L^2(\Lambda)$ to $L^2(\Lambda;\R^d)$.

We get that
\[\Phi^1_\eps(X_n^\eps(\cdot))-\Phi^1_\eps(X^\eps_1(\cdot))\le C\sup_n\norm{X_n^\eps(\cdot)}_{L^2(\Lambda)}+C\norm{X^\eps_1(\cdot)}_{L^2(\Lambda)}\]
which is $\mathbbm{P}$-a.s. in $L^\infty(0,T)$ again by estimate \eqref{ito}.

We continue with the term $\Phi^n_\eps(v)-\Phi^n_\eps(u)$. By the definition of the subgradient it is
bounded by $(\nabla\Phi^n_\eps(v),v-u)_{L^2(\Lambda)}$, which is equal to
\[\int_\Lambda\lrbr{a_p^\eps(\nabla R_\eps(v))}{\nabla R_\eps(v-u)}\,\d\xi.\]
Noticing that $r_\eps^p$ is a contraction on $\R^d$, we can use a similar estimate as
in \eqref{jenseneq} to get that the latter is bounded by
\[C+C\norm{\nabla R_\eps(v)}_{L^2(\Lambda;\R^d)}\norm{\nabla R_\eps(v-u)}_{L^2(\Lambda;\R^d)}.\]
Arguing as above, we see that this term is bounded by
\[C+C\sup_n\norm{X_n^\eps(\cdot)}_{L^2(\Lambda)}\norm{X^\eps_1(\cdot)}_{L^2(\Lambda)}+C\norm{X^\eps_1(\cdot)}_{L^2(\Lambda)}^2,\]
which is $\mathbbm{P}$-a.s. in $L^\infty(0,T)$ by estimate \eqref{ito}.
\end{proof}

We take the limit superior in \eqref{soleq} and continue investigating
\[\limsup_{n}\int_0^t\left[\Phi^n_\eps(X^\eps_1(s))-\Phi^n_\eps(X_n^\eps(s))+\Phi^1_\eps(X_n^\eps(s))-\Phi^1_\eps(X^\eps_1(s))\right]\,\d s.\]
By Lemma \ref{fatoulem}, we can apply (reverse) Fatou's lemma such that it is sufficient to prove that
\[\limsup_{n}\left[\Phi^n_\eps(X^\eps_1(s))-\Phi^n_\eps(X_n^\eps(s))+\Phi^1_\eps(X^\eps_n(s))-\Phi^1_\eps(X^\eps_1(s))\right]\le 0\]
$\mathbbm{P}$-a.s. and
for $\d s$-a.e. $s\in[0,T]$. At this point, we apply Lemma \ref{philem} and get that
\begin{align*}
&\limsup_{n}\left[\Phi^n_\eps(X^\eps_1(s))-\Phi^n_\eps(X_n^\eps(s))+\Phi^1_\eps(X_n^\eps(s))-\Phi^1_\eps(X^\eps_1(s))\right]\\
\le&\limsup_{n}\Phi^n_\eps(X^\eps_1(s))-\liminf_{n}\Phi^n_\eps(X_n^\eps(s))
+\limsup_{n}\Phi^1_\eps(X_n^\eps(s))- \Phi^1_\eps(X^\eps_1(s))\\
\le&\Phi^1_\eps(X^\eps_1(s))-\Phi^1_\eps(Z^\eps(s))+\Phi^1_\eps(Z^\eps(s))-\Phi^1_\eps(X^\eps_1(s))\\
=&0,
\end{align*}
$\mathbbm{P}$-a.s. and
for $\d s$-a.e. $s\in[0,T]$.

\paragraph{Final step}

Going back to
\begin{multline*}
\left\Vert X_{p}\left( t\right) -X_1\left( t\right) \right\Vert _{L^2(\Lambda)}\\
\leq \left\Vert X_{p}\left( t\right) -X_{p}^{\varepsilon }\left( t\right)
\right\Vert _{L^2(\Lambda)}+\left\Vert X_{p}^{\varepsilon }\left( t\right)
-X^{\varepsilon }_1\left( t\right) \right\Vert _{L^2(\Lambda)}+\left\Vert X^{\varepsilon
}_1\left( t\right) -X_1\left( t\right) \right\Vert _{L^2(\Lambda)}
\end{multline*}%
$\mathbb{P}$-a.s. and uniformly in $t\in \left[ 0,T\right]$,
we can complete the proof using Steps I--III as follows.
Let $\delta>0$. Pick $\eps_0>0$, independent of $p$, such that the first and
the third term are less than $\delta/3$. Having fixed $\eps_0$ in such a way,
we can pick $p$ such that the second term is less than $\delta/3$.
\end{proof}

\appendix
\section{Some results on variational convergence}\label{appB}

Let $H$ be a separable Hilbert space.
For a proper, convex functional $\Phi:H\to(-\infty,+\infty]$, the \emph{Legendre transform}
$\Phi^\ast$ is defined by
\[\Phi^\ast(y):=\sup_{x\in H}\left[\lrdel{x}{y}_H-\Phi(x)\right],\quad y\in H.\]
For two functionals $F,G:H\to(-\infty,+\infty]$ the \emph{infimal convolution} $F\# G$ is defined by
\[(F\# G)(y):=\inf_{x\in H}\left[F(x)+G(y-x)\right],\quad y\in H.\]
For a proper, convex, l.s.c. functional $\Phi:H\to(-\infty,+\infty]$, for each $\eps>0$, define the
\emph{Moreau-Yosida regularization}
\[\Phi_\eps:=\Phi\#\frac{1}{2\eps}\norm{\cdot}_H^2.\]
$\Phi_\eps$ is a continuous convex function. Also, $\lim_{\eps\searrow 0}\Phi_\eps=\Phi$ pointwise.

It holds that
\begin{equation}\label{phiasteq}
(\Phi_\eps)^\ast=\Phi^\ast+\frac{\eps}{2}\norm{\cdot}_H^2.
\end{equation}
see e.g. \cite[\S 2.2]{Barb2} and \cite[Ch. 3]{A}.

Recall following definition.
\begin{defi}[Mosco convergence]\label{moscodefi}
Let $\Phi^n:H\to(-\infty,+\infty]$, $n\in\N$, $\Phi:H\to(-\infty,+\infty]$
be proper, convex, l.s.c. functionals.
We say that $\Phi^n\xrightarrow[]{M}\Phi$ in the \emph{Mosco sense} if
\[\tag{M1}
\forall x\in H\;\forall x_n\in H,\; n\in\N,\; x_n\rightharpoonup x\;\text{weakly in $H$}:
           \quad\liminf_n \Phi^n(x_n)\ge\Phi(x).\]
\[\tag{M2}
\forall y\in H\;\exists y_n\in H,\; n\in\N,\; y_n\to y\;\text{strongly in $H$}:
           \quad\limsup_n \Phi^n(y_n)\le\Phi(y).\]
\end{defi}

We shall need following theorem.
\begin{thm}\label{Moscothm}
Let $\Phi^n:H\to(-\infty,+\infty]$, $n\in\N$, $\Phi:H\to(-\infty,+\infty]$
be proper, convex, l.s.c. functionals. Then the following conditions are equivalent.
\begin{enumerate}[(i)]
 \item $\Phi^n\xrightarrow[]{M}\Phi$.
 \item $(\Phi^n)^\ast\xrightarrow[]{M}\Phi^\ast$.
 \item $\forall\eps>0$, $\forall x\in H$: $\lim_n\Phi_{\eps}^n(x)=\Phi_\eps(x)$.
\end{enumerate}
\end{thm}
\begin{proof}
See \cite[Theorems 3.18 and 3.26]{A}.
\end{proof}

\begin{cor}\label{cor1}
Suppose that $\Phi^n\xrightarrow[]{M}\Phi$. Then for each $\eps>0$, $\Phi_\eps^n\xrightarrow[]{M}\Phi_\eps$, too.
\end{cor}
\begin{proof}
Suppose that $\Phi^n\xrightarrow[]{M}\Phi$. By Theorem \ref{Moscothm}, $(\Phi^n)^\ast\xrightarrow[]{M}\Phi^\ast$, too.

If we can prove for each $\eps>0$ that $(\Phi^n_\eps)^\ast\xrightarrow[]{M}(\Phi_\eps)^\ast$, we are done by
Theorem \ref{Moscothm}. (M2) in Definition \ref{moscodefi} follows easily, using equation \eqref{phiasteq} and (M2) for $\{(\Phi_n)^\ast\}$ and $\Phi^\ast$.

Let $x_n\in H$, $n\in\N$, $x\in H$ such that $x_n\rightharpoonup x$ weakly in $H$. By \eqref{phiasteq},
weak lower semi-continuity of
the norm
and (M1) in Definition \ref{moscodefi} for $\{(\Phi_n)^\ast\}$ and $\Phi^\ast$ we get that
\begin{multline*}\liminf_n(\Phi_\eps^n)^\ast(x_n)=\liminf_n\left[(\Phi^n)^\ast(x_n)+\frac{\eps}{2}\norm{x_n}_H^2\right]\\
\ge\liminf_n(\Phi^n)^\ast(x_n)+\liminf_n\frac{\eps}{2}\norm{x_n}_H^2\ge\Phi^\ast(x)+\frac{\eps}{2}\norm{x}_H^2=(\Phi_\eps)^\ast(x).
\end{multline*}
\end{proof}

\subsection{The $L^p$-case}\label{SomeResultsSec}

Let $p\in[1,2]$. We define $j^p:\R^d\to\R$ by
$j^{p}\left( x\right):=\frac{1}{p}\left\vert
x\right\vert^{p}$. Obviously, if $p>1$, each $j^p$
is a convex $C^1$-function.
For $\eps>0$,
let
\[j^p_\eps(x):=\inf_{y\in\R^d}\left[j^p(y)+\frac{1}{2\eps}\abs{x-y}^2\right]\]
be its regularization.
For $u\in L^2(\Lambda;\R^d)$, set
\[\Psi^p(u):=\int_\Lambda j^p(u)\,\d\xi.\]
$\Psi^p$ is a continuous convex functional on $L^2(\Lambda;\R^d)$ for each $p\in[1,2]$.

\begin{lem}\label{lem1}
For $\eps>0$, let $\Psi^p_\eps$ be the Moreau-Yosida regularization of $\Psi^p$ in $L^2(\Lambda;\R^d)$. Then
\[\Psi^p_\eps(v)=\int_\Lambda j^p_\eps(v)\,\d\xi\quad\forall v\in L^2(\Lambda;\R^d).\]
\end{lem}
\begin{proof}
Straightforward from \cite[Theorem 14.60]{RockWets}.
\end{proof}

We would like to prove a convergence result, which shall be useful later.
See the appendix for the terminology. Compare also with \cite{AttCom}.
\begin{lem}\label{lem2}
Let $\{p_n\}\subset[1,2]$, $p_0\in [1,2]$ such that $\lim_n p_n=p_0$. Then
\[\Psi^{p_n}\xrightarrow[]{M}\Psi^{p_0}\quad\text{in the Mosco sense in $L^2(\Lambda;\R^d)$.}\]
\end{lem}
\begin{proof}
Let us prove (M1) in Definition \ref{moscodefi} first. Let $u_n\in L^2(\Lambda;\R^d)$, $n\in\N$, $u\in L^2(\Lambda;\R^d)$ such that
$u_n\rightharpoonup u$ weakly in $L^2(\Lambda;\R^d)$.
W.l.o.g. $\liminf_n\Psi^{p_n}(u_n)<+\infty$. Extract a subsequence (also denoted by $\{u_n\}$) such that
\[\liminf_n\Psi^{p_n}(u_n)=\lim_n\Psi^{p_n}(u_n).\]
Let $v\in L^\infty(\Lambda;\R^d)$. Clearly,
\[\lim_n\int_\Lambda\lrbr{u_n}{v}\,\d\xi=\int_\Lambda\lrbr{u}{v}\,\d\xi.\]
Also, by H\"older's inequality,
\[\frac{1}{p_n}\lrabs{\int_\Lambda\lrbr{u_n}{v}\,\d\xi}^{p_n}\le \Psi^{p_n}(u_n)
\times\left\{\begin{aligned}
    &\abs{\Lambda}^{p_n-1}\norm{v}_{L^\infty(\Lambda;\R^d)}^{p_n},&&\;\;\text{if}\;\;p_0=1,\\                         & \left(\int_\Lambda\abs{v}^{p_n/(p_n-1)}\,d\xi\right)^{p_n-1},&&\;\;\text{if}\;\;p_0>1,                                        \end{aligned}\right.\]
(here $\abs{\Lambda}=\int_\Lambda\,d\xi$).
Upon taking the limit $n\to\infty$, we get that
\[\frac{1}{p_0}\lrabs{\int_\Lambda\lrbr{u}{v}\,\d\xi}^{p_0}\le\liminf_n\Psi^{p_n}(u_n)\times\left\{\begin{aligned}
    &\norm{v}_{L^\infty(\Lambda;\R^d)},&&\;\;\text{if}\;\;p_0=1,\\                         & \left(\int_\Lambda\abs{v}^{p_0/(p_0-1)}\,d\xi\right)^{p_0-1},&&\;\;\text{if}\;\;p_0>1.                                        \end{aligned}\right.
\]
Taking the supremum over all $v\in L^\infty(\Lambda;\R^d)$ with $\norm{v}_{L^\infty(\Lambda;\R^d)}^{p_0/(p_0-1)}\le 1$ and using the
l.s.c. property of the supremum, we get that
\[\Psi^{p_0}(u)=\frac{1}{p_0}\int_\Lambda\abs{u}^{p_0}\,\d\xi\le\liminf_n\Psi^{p_n}(u_n).\]
Since the same argument works for any subsequence of $\{u_n\}$, we have proved (M1).

We are left to prove (M2) in Definition \ref{moscodefi}. Let $u\in L^2(\Lambda;\R^d)$.
Clearly for a.e. $\xi\in\Lambda$
\[\lim_n\frac{1}{p_n}\abs{u(\xi)}^{p_n}=\frac{1}{p_0}\abs{u(\xi)}^{p_0}.\]
But for all $p\in[1,2]$,
\[\frac{1}{p}\abs{u}^p\le 1_\Lambda+\abs{u}^2\in L^1(\Lambda).\]
Hence an application of Lebesgue's dominated convergence theorem yields
\[\lim_n\Psi^{p_n}(u)=\Psi^{p_0}(u).\]
(M2) is proved.
\end{proof}

Theorem \ref{Moscothm}, Corollary \ref{cor1} and Lemmas \ref{lem1}, \ref{lem2} together give:
\begin{cor}
Let $\{p_n\}\subset[1,2]$ such that $\lim_n p_n=1$. Let $\eps>0$.
Then for $u\in L^2(\Lambda;\R^d)$, we have that
\begin{equation}\label{regconveq1}
\lim_n\int_\Lambda j_\eps^{p_n}(u)\,\d\xi=\int_\Lambda j_\eps^{1}(u)\,\d\xi.
\end{equation}
Furthermore, if $u_n\rightharpoonup u$ converges weakly in $L^2(\Lambda;\R^d)$,
we have that
\begin{equation}\label{regconveq2}
\liminf_n\int_\Lambda j_\eps^{p_n}(u_n)\,\d\xi\ge\int_\Lambda j_\eps^{1}(u)\,\d\xi.
\end{equation}
\end{cor}

\end{document}